\documentclass[12pt]{article}
\usepackage[english]{babel}
\usepackage[ansinew]{inputenc}  
\usepackage[T1]{fontenc}
\usepackage{amsmath}
\usepackage{amsfonts}
\usepackage{algorithm, algorithmic}

\usepackage{amsthm}

\usepackage{epsfig}

\usepackage{graphicx}
\usepackage{pstcol,pstricks,pst-node,pst-text,pst-3d,pst-plot}

\def\nref#1{{\rm (\ref{#1})}}

\definecolor{grismeu}{gray}{0.85}

\psfigdriver{dvips}

\title{On the convergence of weighted-average consensus}

\newtheorem{teor}{Theorem}[section]

\newtheorem{coro}{Corollary}[section]
\newtheorem{remark}{Remark}[section]

\author{Francisco Pedroche\thanks{%
Institut de Matem\`{a}tica Multidisciplinària,
Universitat Polit\`{e}cnica de Val\`{e}ncia.
Cam\'{\i} de Vera s/n. 46022 Val\`{e}ncia. Spain.  pedroche@imm.upv.es.}
\and Miguel Rebollo\footnotemark[2]
\and Carlos Carrascosa\footnotemark[2]
\and Alberto Palomares\thanks{%
Departament de Sistemes Informàtics i Computació,
Universitat Polit\`{e}cnica de Val\`{e}ncia.
Cam\'{\i} de Vera s/n. 46022 Val\`{e}ncia. Spain.} }

\date{}

\begin{document}

\maketitle

%habrá que cambiarlo si lo presentamos a ALAMA
\begin{abstract}
In this note we give sufficient conditions for the convergence of the iterative algorithm called
weighted-average consensus in directed graphs. We study the discrete-time form of this algorithm.
We use standard techniques from matrix theory to prove the main result. 
As a particular case one can obtain well-known results for non-weighted average consensus. We also give a corollary for undirected graphs.
\end{abstract}

\noindent
{\bf Keywords:}
Consensus algorithms,
iterative methods, distributed consensus, multi-agent consensus, Perron-Frobenius.
\\

%\tableofcontents

%\pagebreak

\section{Introduction}\label{sec:intro}

%Nascimento et al, 2011 i de wiki
%http://en.wikipedia.org/wiki/Connectivity_(graph_theory)
Let $G=(V,E)$ be a graph, with $V=\{ v_1, v_2, \ldots, v_n \}$
a non-empty set of $n$ vertices (or nodes) and
$E$ a set of $m$ edges. Each edge is defined by the pair
$(v_i,v_j)$, where $v_i, v_j \in V$.  The adjacency matrix
of the graph $G$ is $A = (a_{ij}) \in \mathbb{R}^{n \times n}$ such that
$a_{ij}=1$ if there is a directed edge connecting node $v_i$ to $v_j$, and $0$, otherwise.
We consider directed graphs (digraphs).
The out-degree $d_i$ of a node $i$ is the number of its out-links, i.e.,
$d_i= \sum_{j=1}^{n} a_{ij}$. We define the Laplacian matrix of the graph as $L = D-A$
where $D$ is the diagonal matrix with the out-degrees.
$D=diag(d_{1},d_{2}, \ldots, d_{n})$.

%subgraph and
%$E=\bigcup_{i=1}^{k} E_i$. We call these $G_i$, $i=1,2, \ldots, k$,
%connected components (= (maximal?) connected subgraph).
%Golub, 109
% Horn JOhnson, p 360
We recall that a permutation matrix $F$ is just the identity matrix with its
rows re-ordered. Permutation matrices are orthogonal, i.e., $F^T=F^{-1}$.
A matrix $A \in \mathbb{R}^{n \times n}$, with $n \geq 2$, is said to be
reducible
if there is a permutation matrix $F$ of order $n$ and there is some
integer $r$ with $ 1 \leq r \leq n-1$ such that $
F^T A F = \left[
     \begin{array}{cc}
       B & C \\
       0 & D \\
     \end{array}
   \right]$, where $B \in \mathbb{R}^{r \times r}$, $C \in \mathbb{R}^{r \times (n-r)}$, and
$0  \in \mathbb{R}^{(n-r) \times r}$ is a zero matrix.
A matrix is said to be irreducible if it is not reducible.
It is known (see, e.g., \cite{HoJo}) that the adjacency matrix $A$ of a directed graph
is irreducible if and only if the associated graph $G$ is strongly connected. For an undirected
graph irreducibility implies connectivity. Note that the Laplacian $L=D-A$ is irreducible if and only if $A$ is irreducible.
Note that
$L {\bf e} = {\bf 0}$, with ${\bf e}$ the vector of all ones. $L$ has $0$ as an eigenvalue, and therefore is a singular matrix. Note also that $L$ is irreducible if and only if $G$ is connected.

We use the sum norm (or $l_1$ norm) as the vector norm: $\| {\bf v} \|_1 =
|v_1| + |v_2 | + \ldots |v_n|$.
We denote $N=\{1,2,\ldots,n\}$.

\section{Weighted-average consensus}

Let $G$ be a directed graph. Let ${\bf x}^{0}$ be a (column) vector with the initial state of each node.
Let ${\bf w}=[w_1,w_2, \ldots, w_n]^T$ a vector with the weight associated to each node. The following
algorithm (see \cite{OlFaMu}, p. 225)  can be used to obtain the value of the weighted-average consensus
(that is, a common value for all the nodes, reached by consensus)

\begin{equation}\label{eq:wdotx}
W \dot{x} = - L x
\end{equation}
with $W=diag(w_1,w_2, \ldots, w_n)$,
and $L=D-A$, where $D$ is a diagonal matrix with the out-degrees. A discretized version of \nref{eq:wdotx} is

\begin{equation}\label{eq:xk+1}
x^{k+1}_{i} = x^{k}_{i} + \frac{\epsilon}{w_i} \sum_{j \in N_i}  a_{ij} (x_{j}^k-x_i^k), \quad \forall i \in N
\end{equation}
where $N_i$ denote the set of neighbors of node $i$, that is:
$j \in N_i \longleftrightarrow (v_i,v_j) \in E$. The matrix form of \nref{eq:xk+1} is
\begin{equation}\label{eq:xkP}
{\bf x}^{k+1}  = P_w {\bf x}^{k}  \quad k=0,1,2, \ldots
\end{equation}
where
\begin{equation}\label{eq:P_w}
P_w = I - \epsilon L_w
\end{equation}
where we have denoted $L_w = W^{-1}L$.

From \nref{eq:xk+1} it follows that
\begin{equation}\label{eq:pxo}
{\bf x}^{k}  = P_w^k {\bf x}^{0}, \quad k=1,2, \ldots
\end{equation}

In this note we prove the following

{\bf Theorem}
{\em Let $G$ be a strongly connected digraph.
If $\epsilon < min_{i \in N}(w_i/d_i)$ then the scheme \nref{eq:pxo} converges to the weighted-average
consensus given by
\[
{\bf x_w} = \alpha \, {\bf e}
\]
with ${\bf e}$ the vector of all ones, and $\alpha = \frac{1}{\| {\bf v} \|_1} { {\bf v}^{T} {\bf x}^0}$,
where $ {\bf v}$ is a positive eigenvector of $L_w^T$ associated with the zero eigenvalue, that is}

\[
 L_w^T   {\bf v} = {\bf 0}
\]

\section{Known results}
A matrix $P$ is said to be nonnegative if $P_{ij} \geq 0, \forall (i,j) \in N$. A matrix is said to be
primitive if it is irreducible and has only one eigenvalue of maximum modulus \cite{HoJo}.

The following theorem is known as Perron-Frobenius Theorem.

\begin{teor}[\cite{HoJo}, p. 508]\label{th:perrfro}
If $P \in M_n$ is nonnegative and irreducible, then
\begin{description}
\item[a)] $\rho(P) >0$
\item[b)] $\rho(P)$ is an eigenvalue of $P$
\item[c)] There is a positive vector ${\bf x}$ such that $P{\bf x}=\rho(P){\bf x}$
\item[d)] $\rho(P)$ is an algebraically (and geometrically) simple eigenvalue of $P$
\end{description}

\end{teor}

\begin{teor}[\cite{HoJo}, p. 516]\label{th:nonpri}
If $P \in M_n$ is nonnegative and primitive, then

\[
lim_{k \longrightarrow \infty}[\rho(P)^{-1} P]^{k} = T > 0
\]
where $T={\bf x} {\bf y}^T$, $P{\bf x} = \rho(P) {\bf x}$, $P^{T}{\bf y} = \rho(P) {\bf y}$,
${\bf x} >0$,
${\bf y} >0$, and ${\bf x}^T {\bf y}=1$.

\end{teor}

\begin{remark}\label{rem}
It is known (see \cite{Varga}, p. 48) that given $P$ an irreducible nonnegative matrix then
if $P$ has at least a diagonal entry positive, then $P$ is primitive (in fact, the index of primitivity is related
with the number of diagonal entries positive).  Therefore, we can apply theorem \ref{th:nonpri} to matrices
that are irreducible nonnegative with at least a diagonal entry positive.
\end{remark}

\section{Main Result}

In this section we prove the main theorem and a corollary.

\begin{teor}\label{teow}
Let $G$ be a strongly connected digraph.
If $\epsilon < min_{i \in N}(w_i/d_i)$ then the scheme \nref{eq:pxo} converges to the weighted average
consensus given by

\[
{\bf x_w} = \alpha \, {\bf e}
\]
with ${\bf e}$ the vector of all ones, and $\alpha = \frac{1}{\| {\bf v} \|_1} { {\bf v}^{T} {\bf x}^0}$,
where $ {\bf v}$ is a positive eigenvector of $L_w^T$ associated with the zero eigenvalue, that is:

\[
 L_w^T   {\bf v} = {\bf 0}
\]

\begin{proof} Let us begin with the existence of the positive vector ${\bf v}$.
Since $G$ is strongly connected we have that  $L_w^T$ is irreducible and therefore
$P_w^T$ is irreducible nonnegative. Therefore
by Theorem \ref{th:perrfro} there exists a positive vector ${\bf v}$ such that
$P_w^T {\bf v} = \rho(P_w^T) {\bf v}$. And since $P_w^T$ is column stochastic we have
$\rho(P_w^T)=1$. Then
$P_w^T {\bf v} =  {\bf v}$ and it follows that $(I - \epsilon L_w^T){\bf v} = {\bf v}$ and therefore
$L_w^T {\bf v} = {\bf 0}$.

Since  $\epsilon < w_i/d_i$ for some $i \in N$ we have that $P_w$ is an irreducible nonnegative matrix
with at least a diagonal entry positive. Therefore from remark \ref{rem} we have that $P_w$ is primitive and we can apply theorem \ref{th:nonpri} to conclude that

\[
lim_{k \longrightarrow \infty}[P_w]^{k} = T > 0
\]

where $T={\bf x} {\bf v}^T$, $P_w{\bf x} = {\bf x}$, $P_w^{T}{\bf v}= {\bf v}$,
${\bf x} >0$,
${\bf v} >0$, and ${\bf x}^T {\bf v}=1$.

Since $P_w$  is row stochastic we have that
$P_w {\bf e} = {\bf e}$, therefore we can write
\[
{\bf x} = \alpha \, {\bf e}
\]
for some $\alpha >0$. But the condition  ${\bf x}^T {\bf v}=1$ leads to
$
 \alpha \, {\bf e}^T \, {\bf v}=1
$, that is
\[
 \alpha = \frac{1}{ \| {\bf v} \|_1}
\]
And therefore

\[
T={\bf x} {\bf v}^T = \alpha {\bf e} {\bf v}^T
\]

And then the weighted average consensus is given by

\[
{\bf x_w} = \lim_{k \longrightarrow \infty}[P_w]^{k} {\bf x}^0 = \alpha {\bf e}   {\bf v}^T {\bf x}^0
= \alpha   {\bf v}^T {\bf x}^0 {\bf e}
\]
and the proof follows.
\end{proof}
\end{teor}

{\bf Remark} Note that, as a particular case,  taking the weights ${\bf w}={\bf e}$ this theorem gives theorem 2 of
\cite{OlFaMu}.

In the symmetric case (undirected graph) we have the following

\begin{coro}
Let $G$ be a strongly connected undirected graph.
If $\epsilon < min_{i \in N}(w_i/d_i)$ then the scheme \nref{eq:pxo} converges to the weighted-average
consensus given by  ${\bf x_w}=\alpha \, {\bf e}$ with

\[
\alpha = \frac{1}{\| {\bf w} \|} {\bf w}^T {\bf x}^0   =  \frac{\sum_{i} w_i x_i^0}{  \sum_{i} w_i  }
\]
\end{coro}

\begin{proof}
From theorem \ref{teow} we have that
\[
 L_w^T   {\bf v} = {\bf 0}
\]
that is
$(W^{-1}L)^T {\bf v} = {\bf 0}$ and therefore
\[
L^T W^{-1}  {\bf v} = {\bf 0}
\]
since $W^{-1}$ is a diagonal matrix. Now,
since $L$ is a symmetric matrix we have
\[
L W^{-1}  {\bf v} = {\bf 0}
\]
which means
\[
L  \left[
     \begin{array}{c}
       v_1/w_1 \\
       v_2/w_2 \\
       \vdots\\
       v_n/w_n  \\
     \end{array}
   \right]
= {\bf 0}
\]
and since $\beta \,  {\bf e}$ for $\beta \in \mathbb{R}$ is the eigenspace associated to $\lambda=0$
we have that

\[
 \left[
     \begin{array}{c}
       v_1/w_1 \\
       v_2/w_2 \\
       \vdots\\
       v_n/w_n  \\
     \end{array}
   \right] = \beta  \left[
     \begin{array}{c}
       1 \\
       1 \\
       \vdots\\
       1  \\
     \end{array}
   \right], \forall \beta \in \mathbb{R}
\]
which means
\[
{\bf v} = \beta {\bf w}
\]
Therefore from theorem \ref{teow} we have that the value of the
weighted-average
consensus is ${\bf x_w}=\alpha \, {\bf e}$ with

\[
\alpha =
\frac{1}{\| {\bf v} \|_1} { {\bf v}^{T} {\bf x}^0} =
\frac{1}{\| \beta {\bf w} \|_1} { \beta {\bf w}^{T} {\bf x}^0} =
\frac{1}{\|  {\bf w} \|_1} {  {\bf w}^{T} {\bf x}^0}.
\]

\end{proof}

\section{Conclusion}
In this note we provide sufficient conditions for the convergence of the so-called weighted-average consensus in discrete form for directed graphs. This algorithm is described in \cite{OlFaMu}, but   to our knowledge no sufficient conditions for the
convergence of this algorithm has been already published in the literature. As a particular case our result gives
known-results on non-weighted consensus.  We also provide a corollary for undirected graphs.
These algorithms are commonly used in multi-agent systems.

\section*{Acknowledgments} \noindent This work is supported by
Spanish DGI grant MTM2010-18674,
Consolider Ingenio CSD2007-00022,
PROMETEO 2008/051,
OVAMAH TIN2009-13839-C03-01, and
PAID-06-11-2084.

\end{document}